\documentclass[12pt]{article}
\usepackage{amsmath, amsthm, amssymb}

\newtheorem{theorem}{Theorem}[section]
\newtheorem*{theorem:mcg}{Theorem~\ref{theorem:mcg}}
\newtheorem*{theorem:main}{Theorem~\ref{theorem:main}}
\newtheorem*{theorem:main2}{Theorem~\ref{theorem:main2}}
\newtheorem{lemma}[theorem]{Lemma}
\newtheorem{corollary}[theorem]{Corollary}
\newtheorem{conjecture}[theorem]{Conjecture}
\newtheorem{proposition}[theorem]{Proposition}
\theoremstyle{definition}
\newtheorem{definition}[theorem]{Definition}

\newtheorem{question}[theorem]{Question}

\theoremstyle{remark}
\newtheorem{remark}[theorem]{Remark}
\numberwithin{equation}{section}

\title{$C^1$ actions of the mapping class group on the circle}
\author{Kamlesh Parwani}

\begin{document}

\maketitle

\begin{abstract}
Let $S$ be a connected orientable surface with finitely many punctures, finitely many boundary components, and genus at least $6.$ Then any $C^1$ action of the mapping class group of $S$ on the circle is trivial.

The techniques used in the proof of this result permit us to show that products of Kazhdan groups and certain lattices cannot have $C^1$ faithful actions on the circle. We also prove that for $n  \geq 6$, any $C^1$ action of $Aut(F_n)$ or $Out(F_n)$ on the circle factors through an action of $\mathbb{Z}/2\mathbb{Z}.$ 
\end{abstract}

\section{Introduction} 

\begin{theorem} \label{theorem:mcg}
Let $S$ be a connected orientable surface with finitely many punctures, finitely many boundary components, and genus at least $6.$ Then any $C^1$ action of the mapping class group of $S$ on the circle is trivial.
\end{theorem}

Thurston has established that there are faithful (effective) $C^0$ actions of the mapping class group of $S$ on the circle when $S$ is a surface of negative Euler characteristic with nonempty boundary (see \cite{SW} or \cite{Ghys}). Dehornoy in \cite{Dehornoy}, using different techniques, had previously shown that braid groups (mapping class groups of the punctured sphere) act faithfully on the circle. The fact that mapping class groups act on the circle had been observed earlier by Nielsen in his clasical works. Theorem~\ref{theorem:mcg} above asserts that nontrivial actions do not exist under the smoothness assumption when the genus is sufficiently large.

In fact, any $C^0$ action of a finitely generated group on the circle is (topologically) conjugate to a Lipschitz action (see \cite{DNK}).  So the known actions may be considered to be Lipschitz, and therefore, can be assumed to be differentiable almost everywhere.  The result above shows that when the genus is at least $6$, the actions of mapping class groups are not smoothable, that is, they are not conjugate to $C^1$ actions.

Theorem~\ref{theorem:mcg} provides an infinite family of examples that answer the following question posed by John Franks.

\begin{question}[Franks \cite{Franks}]
Does there exists a finitely generated group which acts faithfully on the circle via homeomorphisms but for which there is no faithful $C^1$ action?
\end{question}

Note that another example that answers this question has been constructed  in \cite{Calegari}.
Theorem~\ref{theorem:mcg} also generalizes the following result in \cite{Farb&Franks}.

\begin{theorem} [Farb and Franks]
Let $S$ be compact surface of genus at least $3$ and at most one puncture. Then any $C^2$ action of the mapping class group of $S$ on the circle must be trivial.
\end{theorem}

The main results of this paper are the following two theorems.

\begin{theorem} \label{theorem:main}
Let $H$ and $G$ be two finitely generated groups such that $H_1(G, \mathbb{Z}) = H_1(H, \mathbb{Z}) = 0.$ Then for any $C^1$ action of $H \times G$ on the circle, either $H \times 1$ acts trivially or $1 \times G$ acts trivially.
\end{theorem}

\begin{theorem} \label{theorem:main2}
Let $H$ and $G$ be two finitely generated groups such that $H^1(G_0,\mathbb{R}) = H^1(H_0,\mathbb{R}) = 0$ for all finite index normal subgroups $H_0$ and $G_0$ of $H$ and $G$ respectively. Then for any $C^1$ action of $H \times G$ on the circle, the induced action of either $H \times 1$ or $1 \times G$ factors through an action of a finite group.
\end{theorem}

Theorem~\ref{theorem:mcg} follows directly from Theorem~\ref{theorem:main}. Theorem~\ref{theorem:main2} has implications for $C^1$ actions of finite index subgroups of mapping class groups. Even if the mapping class groups don't have property (T) (see \cite{T}), it is still conceivable (and conjectured) that all finite index subgroups have trivial first cohomology.  If this is true, Theorem~\ref{theorem:main2} would imply that the following conjecture is true when the genus of surface is at least 4.

\begin{conjecture} \label{conjecture}
Let $G$ be a finite index subgroup of the mapping class group of $S$, where $S$ is a connected orientable surface with finitely many punctures, finitely many boundary components, and genus at least $2$.  Then any $C^1$ action of $G$ on the circle cannot be faithful.
\end{conjecture}

Theorem~\ref{theorem:main} also implies that for $n  \geq 6$, any $C^1$ action of $Aut(F_n)$ or $Out(F_n)$ on the circle factors through an action of $\mathbb{Z}/2\mathbb{Z}.$ Theorem~\ref{theorem:main2} can be applied to show that products of Kazhdan groups and certain lattices cannot have $C^1$ faithful actions on the circle. The proofs of the main results are given in Section~\ref{section:proofs} and the other corollaries are discussed in greater detail in Section~\ref{section:corollaries}.

Conversations with Benson Farb and John Franks were valuable in the preparation of this article. The author would also like to thank the referee for several useful comments and suggestions.

\section{Tools}

We now present several results that will be used in the proof of Theorem~\ref{theorem:mcg} and Theorem~\ref{theorem:main}.

\subsection{Algebraic properties of $Mod(S)$} \label{subsection:split}

Throughout this paper, unless otherwise stated, by a surface we mean a connected orientable surface with finitely many punctures and boundary components. The mapping class group of $S$, denoted by $Mod(S)$, is the group of isotopy classes of orientation preserving diffeomorphisms of $S$; the diffeomorphisms and the isotopies  fix all the punctures and fix the boundary pointwise.
 
When $S$ is a surface of genus at least $6$, there exists a simple closed separating curve  that splits $S$ into two subsurfaces $M_1$ and $M_2$ such that the genus of each of these subsurfaces is at least 3. So $Mod(S)$ contains the subgroup $Mod(M_1) \times Mod(M_2).$ We will show that either $Mod(M_1)$ acts trivially or $Mod(M_2)$ acts trivially in order to establish Theorem~\ref{theorem:mcg}. 

It is known that if $M$ is a surface of genus $2$, $H_1(Mod(M), \mathbb{Z}) = \mathbb{Z}_{10}$, and if $M$ has genus greater that 2, then $H_1(Mod(M), \mathbb{Z}) = 0.$ We will also use the fact that if the genus of $M$ is at least 2, then $Mod(M)$ is generated by Dehn twists about finitely many non-separating simple closed curves. The survey article \cite{Korkmaz} is a good reference for these results.

\subsection{Thurston's Stability Lemma}

\begin{theorem}[Thurston \cite{Thurston}] \label{theorem:Thurston}
Let $G$ be a finitely generated group acting on $\mathbb{R}^n$ with a global fixed point $x.$ If the action is $C^1$ and $Dg(x)$ is the identity for all $g \in G$, then either there is a nontrivial homomorphism of $G$ into $\mathbb{R}$ or $G$ acts trivially.
\end{theorem}

We will need the following result which is a direct consequence of Thurston's Stability Lemma (Theorem~\ref{theorem:Thurston}).

\begin{lemma} \label{lemma:trivial}
Let $G$ be a finitely generated group with an orientation preserving action on the circle. If $G$ acts with a global fixed point and $H^1(G, \mathbb{R}) = 0$, then $G$ acts trivially.
\end{lemma}
\begin{proof}
Let $x$ be the global fixed point. It suffices to show that $g'(x) = 1$ for all $g \in G.$ So consider the homomorphism $L : G \to \mathbb{R}$ defined by $L(g) = \log (g'(x))$. Since $H^1(G, \mathbb{R}) = 0$, this must be the trivial homomorphism, which implies that $g'(x) = 1$ for all $g \in G.$
\end{proof}

\subsection{Rotation numbers}

The subject mater of this subsection is well known.  The interested reader may refer to \cite{Ghys} or even \cite{Farb&Franks} for more details.

\begin{definition}
Let $G$ be a finitely generated subgroup of the orientation preserving homeomorphisms of the circle and let $\mu$ be a $G$-invariant probability measure.
The \textit{mean rotation number homomorphism} is the map $\rho: G \to \mathbb{R}/\mathbb{Z}$ defined by
\[
g \to \int_{S^1} (\tilde{g} - Id) \, d \tilde{\mu} \, \, (\textrm{mod} \, 1), 
\]
where $\tilde{g}$ and $\tilde{\mu}$ are lifts of $g$ and $\mu$ to the real line and the integral is over a single fundamental domain. 
\end{definition}

The fact that this map is a homomorphism follows easily from the assumption that $G$ preserves $\mu.$ Note that the mean rotation number of a homeomorphism is the same as the translation number of a circle homeomorphism, which was originally defined by Poincar\'e.

\begin{proposition}
Let $G$ be a finitely generated group with an orientation preserving action on the circle and let $\mu$ be a $G$-invariant probability measure. The element $g \in G$ acts with a fixed point if $\rho(g) = 0.$
\end{proposition}

\begin{corollary}
Let $G$ be a finitely generated group with an orientation preserving action on the circle and let $\mu$ be a $G$-invariant probability measure. If $H_1(G,\mathbb{Z}) = 0$, then $G$ acts with a global fixed point.
\end{corollary}
\begin{proof}
Since $G$ is perfect and the target of  $\rho$, the rotation number homomorphism, is an abelian group, $\rho$ must be trivial. The above proposition implies that every element acts with a fixed point. The support of $\mu$ is contained in the fixed point set of each element in $G$, and therefore, $G$ must have a global fixed point.
\end{proof}

\begin{corollary} \label{corollary:finite}
Let $G$ be a finitely generated group with an orientation preserving action on the circle and let $\mu$ be a $G$-invariant probability measure. If $H^1(G,\mathbb{R}) = 0$, then $G$ has a periodic orbit. In particular, a finite index normal subgroup of $G$ acts with a global fixed point.
\end{corollary}
\begin{proof}
The assumption $H^1(G,\mathbb{R}) = 0$ implies that the image of the mean rotation number homomorphism is a finite group, and therefore, the kernel $K$ is a finite index normal subgroup. The proposition above implies that every element in $K$ acts with a fixed point. Recall that the support of the invariant measure is contained in the fixed point sets of every element in $K.$ So $K$, a finite index normal subgroup of $G$, acts with a global fixed point.
\end{proof}
\subsection{Hyperbolic fixed points}

The proof of Theorem~\ref{theorem:main} relies heavily on the following result, which guarantees the existence of an element with a finite fixed point set in the absence of an invariant probability measure.

\begin{theorem}[Deroin, Kleptsyn, and Navas \cite{DNK}]
Let $G$ be a countable group with an orientation preserving $C^1$ action on the circle.  If there is no $G$-invariant probability measure for the action, then there exists an element $g \in G$ that only has hyperbolic fixed points. In particular, $g$ has a nonempty finite set of fixed points. 
\end{theorem}

\section{$C^1$ actions on the circle} \label{section:proofs}

\begin{theorem:main}
Let $H$ and $G$ be two finitely generated groups such that $H_1(G, \mathbb{Z}) = H_1(H, \mathbb{Z}) = 0$. Then for any $C^1$ action of $H \times G$ on the circle, either $H \times 1$ acts trivially or $1 \times G$ acts trivially.
\end{theorem:main}
\begin{proof}
Suppose that there exists a $C^1$ action of $H \times G$ on the circle. Since $H$ and $G$ are both perfect, this action must be orientation preserving. We now consider the induced action of $G$ on the circle. By the result of Deroin, Kleptsyn, and Navas, either there exists a probability measure $\mu$ that is $G$-invariant or there is an element  $g \in G$ that has a finite number of fixed points. We treat these two cases separately.

\smallskip

\noindent
CASE 1: There exists a probability measure $\mu$ that is $G$-invariant.

Now consider the mean rotation number homomorphism.  This has to be trivial, and so, every element acts with a fixed point. The support of the measure $\mu$ is contained in the intersection of the fixed point sets of all the elements.  So $G$ has a global fixed point.   Now apply the corollary to Thurston's Stability Lemma (Lemma~\ref{lemma:trivial}) to conclude that  $G$ acts trivially.

\smallskip

\noindent
CASE 2:  There is an element $g$ in $G$ that has a finite number of fixed points.

In this case, there is a finite set---the set of hyperbolic fixed points---left invariant by the induced action of the group $H.$ This implies that the action of the group $H$ has an invariant measure, and now we may argue as above, with $G$ replaced by $H$, to conclude that $H$ acts trivially.
\end{proof}

The proof of Theorem~\ref{theorem:main2} is almost identical to the one given above.  The assumption of trivial cohomologies rather than trivial homologies guarantees the existence of periodic points and not necessarily fixed points.  These periodic points are fixed points for some finite index normal subgroup, which must act trivially for the same reasons presented above.  The proof can be easily obtained by making minor adjustments to the arguments in the proof of Theorem~\ref{theorem:main} above.

\begin{theorem:main2}
Let $H$ and $G$ be two finitely generated groups such that $H^1(G_0,\mathbb{R}) = H^1(H_0,\mathbb{R}) = 0$ for all finite index normal subgroups $H_0$ and $G_0$ of $H$ and $G$ respectively. Then for any $C^1$ action of $H \times G$ on the circle, the induced action of either $H \times 1$ or $1 \times G$ factors through an action of a finite group.
\end{theorem:main2}
\begin{proof}
Suppose that there exists a $C^1$ action of $H \times G$ on the circle. Let $H_0$ and $G_0$ be the index $2$ subgroups of $H$ and $G$ respectively that have induced orientation preserving actions. We now focus on the action of the group $G_0.$ By the result of Deroin, Kleptsyn, and Navas, either there exists a probability measure $\mu$ that is $G_0$-invariant or there is an element  $g \in G_0$ that has a finite number of fixed points. We treat these two cases separately.

\smallskip

\noindent
CASE 1: There exists a probability measure $\mu$ that is $G_0$-invariant.

Apply Corollary~\ref{corollary:finite} to obtain a finite index normal subgroup $G_1$ that acts with a global fixed point.     Now Thurston's Stability Lemma (Lemma~\ref{lemma:trivial}) implies that  $G_1$ acts trivially, and so, the induced action of $G$ factors through an action of a finite group.

\smallskip

\noindent
CASE 2:  There is an element $g$ in $G_0$ that has a finite number of fixed points.

In this case, there is a finite set---the set of hyperbolic fixed points---left invariant by the induced action of the group $H_0.$ This implies that the action of the group $H_0$ has an invariant measure, and now we may argue as above, with $G_0$ replaced by $H_0$, to conclude that the induced action of $H$ factors through an action of a finite group.
\end{proof}

We are now ready to prove Theorem~\ref{theorem:mcg}, which follows immediately from Theorem~\ref{theorem:main} above.

\begin{theorem:mcg}
Let $S$ be a connected orientable surface with finitely many punctures, finitely many boundary components, and genus at least $6.$ Then any $C^1$ action of the mapping class group of $S$ on the circle is trivial.
\end{theorem:mcg}
\begin{proof}
We may split the surface $S$ up into two subsurfaces $M_1$ and $M_2$ as in Section~\ref{subsection:split} to obtain the action of $Mod(M_1) \times Mod(M_2)$ on the circle. Note that $Mod(M_1)$ and $Mod(M_2)$ are both finitely generated and perfect (see \cite{Korkmaz}). Theorem~\ref{theorem:main} implies that either $Mod(M_1)$ acts trivially or $Mod(M_2)$ acts trivially.  In both cases, a Dehn twist about a non-separating simple closed curve acts trivially. Since all such Dehn twists are conjugate to each other (see \cite{Ivanov} for instance), all of them must act trivially.  Now recall the $Mod(S)$ is generated by these Dehn twists.
\end{proof}

\begin{remark}
Suppose that every finite index normal subgroup of $Mod(S)$ has trivial first cohomology when $S$ is a surface of genus at least 2. Now consider a $C^1$ action of a finite index subgroup $K$ of $Mod(S)$, where $S$ has genus at least 4. As in the proof of the theorem above, we may split $S$ into two subsurfaces $M_1$ and $M_2$, each with genus at least 2. The action of $K$ induces an action of $H \times G$, where $H$ and $G$ are finite index subgroups of $Mod(M_1)$ and $Mod(M_2)$ respectively. Under our assumptions, Theorem~\ref{theorem:main2} implies that the action of $K$ cannot be faithful. This elaborates on the comments just beore Conjecture~\ref{conjecture}.
\end{remark}

The following corollary is a direct consequence of Theorem~\ref{theorem:main}.

\begin{corollary}
For $1 \leq i \leq n$, let $G_i$ be a finitely generated group with $H_1(G_i, \mathbb{Z}) = 0.$ Then for any $C^1$ action of $G_1 \times G_2 \times \dots \times G_n$ on the circle, there exists at most one $G_i$ with a nontrivial induced action on the circle.
\end{corollary}

An analogous statement follows from Theorem~\ref{theorem:main2}.

\section{Other corollaries} \label{section:corollaries}

We first apply Theorem~\ref{theorem:main} to $Aut(F_n)$ and $Out(F_n)$ and then apply Theorem~\ref{theorem:main2} to show that products of Kazhdan groups cannot act smoothly on the circle.

\subsection{$Aut(F_n)$ and $Out(F_n)$}

We now discuss the actions of  $Aut(F_n)$, the automorphism group of $F_n$, and $Out(F_n)$, the outer automorphism group of $F_n$, where $F_n$ is the free group of rank $n.$ The proof that the $C^1$ actions of $Aut(F_n)$ and $Out(F_n)$ on the circle are trivial is similar to the proof of Theorem~\ref{theorem:mcg}.

The lemma below follows directly from the proof of Lemma 4.1 in \cite{Farb&Franks}.

\begin{lemma}
For $n \geq 3$, $Aut(F_n)$ contains a finitely generated subgroup $T$ of index $2$ which has a set of generators $\{A_{ij} , B_{ij} \}$ with $i \neq j$, $1 \leq i \leq n$ and $1 \leq j \leq n.$ The subgroup $T$ has a following properties.

\begin{itemize}

\item $H_1(T, \mathbb{Z}) = 0.$

\item $A_{ij}$ is conjugate to $A_{kl}$ and $B_{ij}$ is conjugate to $B_{kl}$, for all positive integers $i, j, k, l$ between $1$ and $n.$

\end{itemize} 
\end{lemma}

\begin{theorem}
For $n  \geq 6$, any $C^1$ action of $Aut(F_n)$ or $Out(F_n)$ on the circle factors through an action of $\mathbb{Z}/2\mathbb{Z}.$ 
\end{theorem}
\begin{proof}
Suppose $Aut(F_n)$, where $n \geq 6$, has a $C^1$ action on the circle. We may consider the induced action of $Aut(F_3) \times Aut(F_{n-3})$, which is a subgroup of $Aut(F_n)$, on the circle. Apply the lemma above to obtain an action of $H \times G$ on the circle, where $H$ and $G$ are the index 2 subgroups of $Aut(F_3)$ and $Aut(F_{n-3})$ respectively that play the role of the subgroup $T$ in the lemma above.

Since $H$ and $G$ are finitely generated and perfect, Theorem~\ref{theorem:main} implies that either $H$ acts trivially or $G$ acts trivially. In either case, some $A_{ij}$ and some $B_{kl}$ must act trivially. Now recall that $Aut(F_n)$ has an index two subgroup that is generated by elements conjugate to $A_{ij}$ and  $B_{kl}.$ So an index two subgroup of $Aut(F_n)$ acts trivially.

Since there exists a natural homomorphism from $Aut(F_n)$ onto $Out(F_n)$, the same result holds for $Out(F_n)$ also.
\end{proof}

Note that the result of Bridson and Vogtmann in \cite{Brid&Vogt} about $C^0$ actions implies this theorem.  However, their proof strongly relies on the existence of finite order elements. We give our proof here because it is short, easy, and it has implications for torsion free finite index subgroups of $Aut(F_n)$ and $Out(F_n).$ It should also be noted that Farb and Franks prove the same result for $C^2$ actions in \cite{Farb&Franks}.

\subsection{Kazhdan groups}

It is well known that discrete Kazhdan groups are finitely generated and any finite index normal subgroup of a discrete Kazhdan group has trivial first cohomology. So Theorem~\ref{theorem:main2} implies that products of discrete Kazhdan cannot have faithful $C^1$ actions on the circle.

\begin{corollary}
Let $H$ and $G$ be two discrete Kazhdan groups. Then for any $C^1$ action of $H \times G$ on the circle, the induced action of either $H \times 1$ or $1 \times G$ factors through an action of a finite group.
\end{corollary}

This is related to a result of Navas about actions of groups with property (T) on the circle; in \cite{Navas} he proves that for $\alpha > 1/2$, any $C^{1 +\alpha}$ action of a discrete Kazhdan group on the circle factors through an action of a finite group.

Certain lattices like $SL(n,\mathbb{Z})$ for $n >2$ are known to have Kazhdan's property (T).  Also when $n>5$, $SL(n,\mathbb{Z})$ contains two commuting sub-lattices isomorphic to $SL(3,\mathbb{Z}). $ These facts yield the following corollary.
 
\begin{corollary}
Any $C^1$ action of a finite index subgroup of $SL(n,\mathbb{Z})$, for $n>5$, on the circle  factors through an action of a finite group.
\end{corollary}

This result is well known. It follows from the work of Morris (Witte) in \cite{Witte} (also see \cite{Ghys2} and \cite{Burger&Monod}).  The corollary above holds for a large family of lattices that contain two commuting sub-lattices that satisfy the hypothesis of Theorem~\ref{theorem:main2}.

The interested reader should also see \cite{BFS} and \cite{Navas2} for  results about actions of products of lattices and Kazhdan groups on the circle.

\end{document}